\documentclass[12pt]{amsart}

\title[Entangled quantum states via
geometric invariant theory]{Existence of locally maximally
entangled quantum states \\ via geometric invariant theory}
\author{Jim Bryan, Zinovy Reichstein, and Mark Van Raamsdonk}

\date{\today} \address{
Department of Mathematics\\
University of British Columbia \\
Room 121, 1984 Mathematics Road  \\
Vancouver, B.C., Canada V6T 1Z2
}

\address{
Department of Physics and Astronomy\\
University of British Columbia \\
6224 Agricultural Road  \\
Vancouver, B.C., Canada V6N 3T8
}

\thanks{Bryan and Reichstein were partially supported
by grants from the National Science and Engineering Council
of Canada (NSERC).  Van Raamsdonk was partially supported
by grants from NSERC and the Simons Foundation.}

\usepackage{graphicx}
\usepackage{verbatim}
\usepackage{amsmath,amsthm,amsfonts}
\usepackage{amssymb}
\usepackage{times}
\usepackage{enumitem}
\usepackage[all]{xy}
\usepackage{esvect}
\newtheorem{theorem}{Theorem}[section]
\newtheorem{proposition}[theorem]{Proposition}

\newtheorem{lemma}[theorem]{Lemma}
\newtheorem{corollary}[theorem]{Corollary}

\theoremstyle{definition}

\newtheorem{def-theorem}[theorem]{Definition-Theorem}
\newtheorem{remark}[theorem]{Remark}
\newtheorem{definition}[theorem]{Definition}
\newtheorem{example}[theorem]{Example}

\newcommand{\be}{\begin{equation}}
\newcommand{\ee}{\end{equation}}
\newcommand{\bea}{\begin{eqnarray}}
\newcommand{\eea}{\end{eqnarray}}
\newcommand{\beas}{\begin{eqnarray*}}
\newcommand{\eeas}{\end{eqnarray*}}
\newcommand{\ba}{\begin{array}}
\newcommand{\ea}{\end{array}}

\newcommand{\CC} {\mathbb{C}}          
\newcommand{\PP} {\mathbb{P}}
\renewcommand{\AA} {\mathbb{A}}

\newcommand\INTO{\ar@{^{(}->}[r]}

\newcommand{\Lie}{\operatorname{Lie}}

\newcommand{\Span}{\operatorname{Span}}
\newcommand{\Mat}{\operatorname{M}}

\newcommand{\GL}{\operatorname{GL}}
\newcommand{\SL}{\operatorname{SL}}
\newcommand{\SU}{\operatorname{SU}}
\newcommand{\Tr}{\operatorname{Tr}}
\newcommand{\tr}{\operatorname{tr}}

\newcommand{\trdeg}{\operatorname{trdeg}}

\newcommand{\Proj}{\operatorname{Proj}}

\renewcommand{\emptyset}{\varnothing}
\renewcommand{\hat}{\widehat}

\newcommand{\identity}{1\!\!1}
\newcommand{\GITquot}{/\!/}
\newcommand{\lcm}{\operatorname{lcm}}
\newcommand{\dvec}{\mathbf{d}}
\newcommand{\avec}{\mathbf{a}}
\newcommand{\evec}{\mathbf{e}}

\newcommand{\gmax}{g_{\mathsf{max}}}

\newcommand{\presectionspace}{\vspace{0.2cm}} 

\begin{document}

\begin{abstract}

We study a question which has natural interpretations both in quantum
mechanics and in geometry. Let $V_{1},\dotsc , V_{n}$ be complex vector
spaces of dimension $d_{1},\dots ,d_{n}$ and let $G= \SL_{d_{1}}
\times \dots \times \SL_{d_{n}}$. Geometrically, we ask: Given
$(d_{1},\dots ,d_{n})$, when is the geometric invariant theory
quotient $\mathbb{P}(V_{1}\otimes \dots \otimes V_{n})\GITquot G$
non-empty?  This is equivalent to the quantum mechanical question of
whether the multipart quantum system with Hilbert space $V_{1}\otimes
\dots \otimes V_{n}$ has a locally maximally entangled state, i.e. a
state such that the density matrix for each elementary subsystem
is a multiple of the identity. We show that the answer to this
question is yes if and only if $R(d_{1},\dotsc ,d_{n})\geqslant 0$ where
\[
R(d_{1},\dotsc ,d_{n}) = \prod_{i}d_{i} +\sum_{k=1}^{n}
(-1)^{k}\sum_{1\leqslant i_{1}<\dotsb <i_{k}\leqslant n}
\left(\gcd(d_{i_{1}},\dotsc ,d_{i_{k}}) \right)^{2}.
\]
We also provide a simple recursive algorithm which determines the answer to the
question, and we compute the dimension of the resulting quotient in
the non-empty cases.
\end{abstract}

\maketitle

\presectionspace
\section{Introduction}

In a multipart quantum system, the space of pure states is described
by a tensor product Hilbert space
\[
V = V_{1}\otimes \dots \otimes V_{n},
\]
where $V_{i}$ are $d_i$-dimensional Hilbert spaces describing the
elementary subsystems in isolation\footnote{As usual, we require the
states $\psi$ to be normalized $(\psi,\psi) = 1$ and identify states
which are related by an overall phase $\psi \sim e^{i \theta}
\psi$. Thus, we can identity states $\psi$ with points in the
projective space $\mathbb{P}(V)$.}$^{,}$\label{page one}\footnote{\label{footnote:
allow di=1 but need two non-trivial di}It is natural to assume
that $d_i \geqslant 2$ for every $i$; however, we will allow trivial
subsystems (with $d_i = 1$), as long as there at least two subsystems
of dimension $d_i \geqslant 2$.  }.
Given a pure state $\psi \in V$, the associated state of the $i$th
elementary subsystem is described by the {\it reduced density
operator} $\rho_i : V_i \to V_i$, a nonnegative unit-trace Hermitian
operator defined by the action of the contraction map $V_1 \otimes
\cdots \otimes V_n \otimes V_1^* \otimes \cdots \otimes V_n^* \to V_i
\otimes V_i^*$ on the operator $\rho_\psi = \psi \otimes \psi^* \in V
\otimes V^*$. Equivalently, we have $\rho_i = \tr_{V_1 \otimes \cdots
\otimes \hat{V}_i \otimes \cdots \otimes V_n} \rho_\psi$.

In general, the structure of entanglement in a multipart quantum
system is related to the eigenvalue spectra of the reduced density
matrices for subsystems. A subsystem $i$ is {\it entangled} with the
rest of the system if its spectrum is different from $\{1,0,0,\dots\}$
(i.e. if its density operator is not a projection operator associated
with a single state). A subsystem $i$ is {\it maximally entangled}
with the rest of the system if all eigenvalues of $\rho_i$ are equal
i.e. $\rho_i = \identity/d_i$. For some $\{d_1,\cdots,d_n\}$ it is
possible to find states where each elementary subsystem has maximal
entanglement:
\begin{definition}
A pure state $\psi$ in a multipart quantum system described by Hilbert
space $V_{1}\otimes \dots \otimes V_{n}$ is said to be {\it locally
maximally entangled} if for each elementary subsystem $i$, the density
operator $\rho_i$ is a multiple of the identity operator, $\rho_i =
\identity/d_i$.
\end{definition}
States that are locally maximally entangled (LME) have many
applications in the field of quantum information theory and quantum
computing. For example, this property is present in Bell states, GHZ
states, quantum error correcting code states, cluster states, and
graph states.

In this paper, our goal is to understand for which
$\{d_1,\cdots,d_n\}$ such LME states exist, and in those cases to
characterize the subset $V_{LME} \subset V$ of such states in the full
Hilbert space.  While the necessary conditions $d_i \leqslant \prod_{j
\ne i} d_j$ are well known and have been suggested to be sufficient,
we will see that the necessary and sufficient condition is
significantly more complicated.

\subsection*{Relation to geometry:}

Remarkably, the quantum mechanics problem we have described is equivalent
to two very natural problems in geometry, the first related to symplectic
geometry and the second related to algebraic geometry and geometric
invariant theory; see A.~A.~Klyachko~\cite{Kly02}, \cite[\S~3]{Kly07},
and N.~R.~Wallach~\cite[\S~4]{wallach}.

Geometrically, the full space of pure states $\mathbb{P}(V)$ admits
a natural symplectic structure defined by the Fubini-Study
symplectic form. Consider
the action of $K = \SU_{d_1} \times \dots \times \SU_{d_n}$
on $\mathbb{P}(V)$, where $\SU_{d_i}$ acts via the fundamental
representation on $V_i$. The spectra of the density operators
$\rho_i$ are invariant under the action of $K$, so we can group
LME states into equivalence classes defined by the $K$ orbits.
Since the symplectic form is invariant under the action of $K$,
we can define a moment map $\mu : \mathbb{P}(V) \to \mathfrak{k}^*$
from our space to the dual of the Lie algebra of $K$. The relation
to symplectic geometry is provided by the result that
$\psi \in V_{LME}$ if and only if $\mu(\psi) = 0$ (see
section \ref{sec:background}). Thus, the space of equivalence
classes of LME states under $K$ is precisely the symplectic
quotient $\mu^{-1}(0)/K$.

By the Kempf-Ness theorem,\footnote{For a discussion of geometric
invariant theory, symplectic geometry, and the Kempf-Ness theorem, see
\cite{mfk}, or see \cite{Hoskins} for a pedagogical introduction.}
this symplectic quotient is equivalent to the geometric invariant
theory quotient $\mathbb{P}(V) \GITquot G$, where
\[
G = \SL_{d_{1}} \times \dots \times \SL_{d_{n}}
\]
is the complexification of $K$. Thus
for dimensions $\{d_1,\cdots,d_n\}$, we arrive at the central observation:
\bigskip
\begin{quote}
{\it There exist LME states if
and only if the quotient $\mathbb{P}(V)\GITquot G$ is non-empty.}
\end{quote}
\bigskip

Our main result provides a simple numerical criterion for when the
quotient $\PP (V) \GITquot G$  is non-empty, and we give a formula for
the dimension of the non-empty quotients. For
\[
\dvec  = (d_{1},\dotsc ,d_{n})
\]
we define the \emph{expected dimension}
\begin{align*}
\Delta (\dvec ) & = \dim \PP (V) -\dim G\\
& = \prod_{i} d_{i} -1 - \sum_{i} (d_{i}^{2}-1)
\end{align*}
as well as the arithmetic functions
\[
\gmax (\dvec ) = \max_{1\leqslant i<j\leqslant n}(\gcd (d_{i},d_{j}))
\]
and
\[
R(\dvec  )  = \prod_{i} d_{i} +\sum_{k=1}^{n}
(-1)^{k} G_{k}(\dvec  ) \, ,
\]
where
\[
G_{k}(\dvec ) = \sum_{1\leqslant i_{1}<\dotsb <i_{k}\leqslant n}
\left(\gcd \left(d_{i_{1}},\dotsc ,d_{i_{k}} \right) \right)^{2} \, .
\]

Our main theorem is the following:

\begin{theorem}\label{thm:main.theorem,numerical.formulation}
The GIT quotient $\mathbb P(V) \GITquot G$ is non-empty if and only if
$R(\dvec )\geqslant 0$. Moreover:

If $\Delta(\dvec) > -2$, then $R >  0$ and $\dim \PP (V)\GITquot G =
\Delta(\dvec) >  0$.

If $\Delta(\dvec) = -2$, then $R > 0$ and $\dim \PP (V)\GITquot G =
\max (\gmax(\dvec) -3,0)$.

If $\Delta(\dvec) < -2$, then $R \leqslant 0$ and the quotient is a single point for $R=0$ and empty for $R<0$.
\end{theorem}

We remark that Theorem~\ref{thm:main.theorem,numerical.formulation}
implies that the quotient is of the expected dimension $\Delta(\dvec
)$ whenever this is non-negative. If the expected dimension is
negative, then the quotient can be empty, a single point, or positive
dimensional.

Our proof of Theorem~\ref{thm:main.theorem,numerical.formulation} is based
on solving the following simple algorithm which
computes the dimension of $\PP (V)\GITquot G$.

\begin{theorem}\label{thm.GIT}
The GIT quotient $\mathbb P(V) \GITquot G$ has
dimension $D(\dvec )$, where $\dvec  = (d_1, \dots, d_n)$ and
$D$ is a function defined on weakly increasing tuples
of integers by the following cases, depending on the
size of $d_n$ relative to $P = d_1 \cdots d_{n-1}$.
\be
\nonumber
\begin{array}{ll}
(a) \; d_n > P & \qquad  D(\dvec ) = -1 . \cr
      &      \cr
(b) \;  d_n = P & \qquad D(\dvec ) = 0 . \cr
      &      \cr
(c) \; \frac{P}{2}
< d_n < P & \qquad D(\dvec ) =
D({\rm sort}(d_1, \dots , d_{n-1}, P - d_n)) . \cr
      &      \cr
(d) \;  d_{n-1} \leqslant d_n \leqslant {\frac{P}{2}} & \qquad  D(\dvec ) = \left\{
\begin{array}{ll} 0 & \dvec  = (1,\dots,1,2,2,2) \cr d-3 & \dvec  = (1,\dots,1,2,d,d) , d \geqslant 3 \cr \Delta (\dvec ) & {\rm otherwise.}
\end{array}
\right.
\end{array}
\ee
In particular, the GIT quotient $\mathbb P(V) \GITquot  G$
is empty if and only if $D(\dvec )=-1$ and
is a single point if and only if $D(\dvec )=0$.
\end{theorem}
In case (c), $D(\dvec )$ is defined recursively
in terms of the value of $D$ on another dimension vector
${\rm sort}(d_1, \dots , d_{n-1}, P - d_n)$, obtained by arranging
the positive integers
\[ d_1, \dots , d_{n-1}, P - d_n \]
in weakly increasing order.  Since the sum of these integers
is strictly less than $\sum_i d_i$, the recursion stops
after at most $\sum_i d_i$ steps. In particular, the function
$D(\dvec )$ is well defined.

Note that the condition that $\PP(V) \GITquot G = \emptyset$
is equivalent to $V$ being a pseudo-homogeneous space for $G$;
see Corollary~\ref{cor1}.  Recall that an $H$-representation $W$
is called a pseudo-homogeneous space if $H$ has a Zariski dense
orbit in $W$.  Pseudo-homogeneous spaces for reductive groups
were classified by M.~Sato\ and\ T.~Kimura~\cite{sato-kimura}.
The vectors $\dvec$ such that $\PP (V) \GITquot G = \emptyset$
can, in principle, be described by appealing to this classification.
Note that the passage from $V$ to $V'$ in Theorem~\ref{thm.GIT}(c)
is an example of what Sato and Kimura called a castling operation.

Alternatively,  the vectors $\dvec$ such that
$\PP(V) \GITquot G = \emptyset$ (or equivalently,
$V \GITquot G$ is a point) can also be described by appealing to
the classification of coregular irreducible representations
$\phi \colon H \to \GL(W)$ such that $H$ is semisimple,
due to P.~Littelmann~\cite{littelmann}.
Here $\phi$ is called coregular if $W \GITquot H$ is an affine space $\AA^m$.
(We are only interested in the cases
where $H = \SL_{d_1} \times \dots \times \SL_{d_n}$,
$V = V_{d_1} \otimes \dots \otimes V_{d_n}$ and $m = 0$.)
Note that the castling operation plays a prominent role
in~\cite{littelmann} as well.

\begin{example} \label{ex.n=2}
For $n=2$, the GIT quotient $\mathbb P(V) \GITquot  G$ is non-empty
if and only if $(d_1,d_2) = (d,d)$.
Indeed,  if $d_2 > d_1$, then part (a) of Theorem~\ref{thm.GIT}
applies, and $\mathbb P(V) \GITquot  G = \emptyset$. If $d_2 = d_1$, then
part (b) applies, and $\mathbb P(V) \GITquot  G$ is single point.
\qed
\end{example}

For $n \geqslant 3$, the situation is considerably more complicated.
For example, the quotient for dimension vectors of the form $(2, d_1,
d_2)$ is nonempty iff $d_1 = d_2 \geqslant 2$ or $d_2/d_1 = (k+1)/k$
with integer $k \geqslant 2$; see Corollary~\ref{cor.d=2}.  The
general characterization of the set of dimension vectors
$(d_{1},d_{2},d_{3})$ which admit LME states is arithmetically
complicated and can be described in terms of generalized Fibonacci
sequences, see~\cite{sam}.

The remainder of this paper is structured as follows. In section
\ref{sec:background}, we provide some additional background on the
quantum mechanics problem and its connection to symplectic geometry and
geometric invariant theory.  Theorem~\ref{thm.GIT} is proved in
sections \ref{sec:GITprelim} and~\ref{sec:proof}.
Our argument does use the above-mentioned classifications
due to Sato-Kimura and Littelmann; however, the proof of part (d)
relies on the work of A.~G.~Elashvili~\cite{elashvili}.
Sections~\ref{sec:analysis} and~\ref{sect.proof-main}
are devoted to analyzing the recursive
algorithm appearing in Theorem \ref{thm.GIT} and proving
Theorem~\ref{thm:main.theorem,numerical.formulation}.

In a companion paper \cite{sam}, we discuss numerous explicit results
and examples of locally maximally entangled states with a view to
applications in quantum information theory.

\section{Background}\label{sec:background}

We begin with a short review of relevant background material on quantum mechanics, the connection to symplectic geometry, and basics of geometric invariant theory.

\subsection*{Density operators and entanglement}

In a multipart quantum system with Hilbert space $V = V_{1}\otimes \dots \otimes V_{d}$, physical observables associated with subsystem $i$ correspond to Hermitian operators $\mathcal{O} : V_i \to V_i$; the expectation value of
the observable in a measurement on a state $\psi_i \in V_i$ is given
by the inner product $(\psi_i,\mathcal{O} \psi_i)$. Any such observable
can be promoted to an observable in the full multipart system;
the associated Hermitian operator $\hat{\mathcal{O}}$ acting on $V$
is defined by $\hat{\mathcal{O}} = \identity \otimes \dots \otimes
\mathcal{O} \otimes \dots  \otimes \identity$.

A crucial feature of multipart quantum systems is that their states are generally {\it entangled}; that is, they cannot be written as products $\psi_1 \otimes \dots \otimes \psi_n$. Furthermore, for $\psi \in V$, there does not generally exist $\psi_i \in V_i$ for which $(\psi_i, \mathcal{O} \psi_i) = (\psi, \hat{\mathcal{O}} \psi)$ for all $\mathcal{O}$ acting on $V_i$. Thus, in the context of a multipart system it is no longer possible to represent the state of an individual subsystem simply as a vector or {\it pure state} in Hilbert space $V_i$. Rather, the subsystem can be described by a {\it density operator}, defined as a non-negative Hermitian operator $\rho_i: V_i \to V_i$ with unit trace. The density operator $\rho_i= \tr_{V_1 \otimes \cdots  \hat{V}_i \cdots \otimes V_n} \rho_\psi$ defined in the introduction is the unique density operator $\rho_i$ such that $(\psi, \hat{\mathcal{O}} \psi) = \tr(\rho_i \mathcal{O})$ for all $\mathcal{O}$ acting on $V_i$.

A subsystem described by a density operator with
eigenvalues/eigenvectors $\{(p_i, \psi_i)\}$ can be interpreted as
being in a statistical ensemble or {\it mixed state} in which we have
state $\psi_i$ with probability $p_i$. This subsystem is entangled
with the rest of the system unless $\{p_i\} = \{1,0,\dots,0\}$. When the density matrix for the subsystem is a
multiple of the identity operator, $\rho_i = \identity/d_i$, the
subsystem is in an equal mixture of all possible states for the
subsystem and we say that the subsystem is {\it maximally mixed} or
{\it maximally entangled} with the rest of the system. The locally
maximally entangled states that we characterize in this paper are
defined by the condition that all elementary subsystems have this
property.

\subsection*{The quantum marginal problem}

The existence question that we consider is a special case of the {\it quantum marginal problem}: which collections of density operators $\{\rho_\alpha\}$ associated with subsystems $\alpha$ of a multipart system can arise from a quantum state of the entire system? In our case where the subsystems are non-overlapping and the state of the full system is assumed to be pure, a general answer to this question has been provided by Klyachko \cite{Kly04} (see \cite{Walter} for a review) via a set of inequalities on the spectra for the density operators, or equivalently, in terms of a criterion expressed in the language of representation theory of the symmetric group. These results provide an in-principle method to answer our question, but one that quickly becomes computationally intractable as the subsystem dimensions increase.

\subsection*{The moment map}

We now briefly review the connection to symplectic geometry. The
Fubini-Study symplectic form on $\mathbb{P}(V)$ is
fixed up to overall scaling by its invariance under $U(V)$
transformations and is thus invariant under $K = \SU (V_{1})
\times \dots \times \SU (V_n)$. The associated moment map $\mu
:\mathbb{P}(V)\to \mathfrak{k}^{*}$ is given
explicitly by
\be \mu(\psi): k \mapsto (\psi, k \psi) \qquad k \in
\mathfrak{k} \; .  \ee
Any $k$ may be written as a linear combination of elements of the form
$\identity \otimes \cdots k_i \otimes \cdots \identity$ with $k_i$ a
traceless, Hermitian operator acting on $V_i$. For an element of this
form, we have
\be
\mu(\psi)(\identity \otimes \cdots k_i \otimes \cdots \identity) =
\tr(\rho_i k_i) \; .
\ee
For $\psi \in V_{LME}$, each $\rho_i$ is proportional to the identity
operator; tracelessness of $k_i$ then implies that the moment map
vanishes. Conversely, $\tr(\rho_i k_i)$ vanishes for arbitrary
traceless Hermitian $k_i$ if and only if $\rho_i$ is proportional to
the identity operator. Thus we have that $V_{LME} = \mu^{-1}(0)$.

\subsection*{The Kempf-Ness theorem and geometric invariant theory}

As discussed above, the space of LME states, up to equivalence, is
given by the symplectic quotient
\[
\mu^{-1}(0)/K.
\]

The Kempf-Ness theorem identifies this space with an algebro-geometric
quotient given by geometric invariant theory which we now briefly
describe. Let
\[
G = 
SL(V_{1})\times \dotsb \times SL(V_{n}).
\]
Note that $K$ is a maximal compact subgroup of $G$.
$G$ acts algebraically on $\PP (V_{1}\otimes \dotsb \otimes V_{n}) =
\PP (V)$ and the geometric invariant theory (GIT) quotient $\PP (V)\GITquot
G$ is the projective variety defined by
\[
\PP (V) \GITquot G = \Proj \left(\CC [V]^{G} \right).
\]
Here $\CC [V]$ is the ring of polynomial functions on $V$ (graded by
degree), and $\CC [V]^{G}$ is the subring of invariant
functions. Concretely, $\Proj \left(\CC [V]^{G} \right)$ can be
constructed as follows (cf.~\cite{Reid02}). It is possible to choose homogeneous
generators $f_0, \dots, f_N$ for the ring of invariants $\mathbb
C[V]^G$ (see~\cite[Thm.~1.1]{mfk}), where we define $w_i$ to be the
degree of $f_i$.  Consider the rational\footnote{Here, rational
indicates that $\pi$ is only defined on a dense open set $U^{ss}
\subset \mathbb P^N$. The points of $U^{ss}$ are called
semistable points.} map
\[
\pi \colon \mathbb P(V) \dasharrow
\mathbb P^N_{w_0, \dots, w_N}
\]
given by $v \mapsto (f_0(v): \cdots :f_N(v))$, where $\mathbb P^N_{w_0, \dots, w_N}$ stands for the $N$-dimensional graded
weighted projective space with weights $w_0, \dots, w_N$. Then the closure of the image of $\pi$ is the GIT quotient
$\PP (V)\GITquot G$, and $\pi$ is the quotient map.
Geometrically, the points of $\PP (V)\GITquot G$ correspond to
closed orbits $\mathcal{O}$ of $G$ in $V \setminus \{ 0 \}$,
up to projective equivalence.

In this context, the Kempf-Ness theorem states that there is a
homeomorphism
\[
\mu^{-1}(0)/K \cong \PP (V)\GITquot G
\]
where $\PP (V)\GITquot G$ is given the complex analytic topology.
The Kempf-Ness theorem thus converts the problem of understanding
$\mu^{-1}(0)/K$, the
space of equivalence classes of LME states, into the purely algebraic
problem of understanding the GIT quotient $\PP (V)\GITquot G$. We
study this quotient in depth in sections ~\ref{sec:GITprelim} and
\ref{sec:proof} .

\section{Invariant-theoretic preliminaries} \label{sec:GITprelim}

\subsection*{Notational conventions}
In the sequel we will denote the $\mathbb C$-algebra of regular functions
on complex affine algebraic variety $X$ by $\mathbb C[X]$.
If $X$ is irreducible (but not necessarily affine), then
$\CC(X)$ will denote the field of rational functions on $X$.

Recall $\dim(X) = \trdeg_{\mathbb C} (\mathbb C(X))$.
Here $\dim(X)$ denotes the dimension of $X$ and
$\trdeg_{\mathbb C} (\mathbb C(X))$ denotes the transcendence degree
of $\CC(X)$ over $\CC$, i.e., the maximal number of elements
$f_1, \dots, f_n \in \CC(X)$ which are algebraically independent over $\CC$.

Finally, if $G$ is a complex algebraic group acting on $X$, we will
denote the ring of $G$-invariant regular functions on $X$ by $\mathbb C[X]^G$
and the field of $G$-invariant rational functions by $\mathbb C(X)^G$.

\subsection*{Stabilizers in general position}
Let $G$ be a reductive complex linear algebraic group and
$\rho \colon G \to \GL(V)$ be a linear representation. By a theorem of
Richardson~\cite{richardson}, the action of $G$ on $V$ has a stabilizer $S$
in general position. That is, there exists a closed
subgroup $S \subset G$ and a $G$-invariant  dense open
subset $U \subset V$ such that the stabilizer $G_x$ is conjugate to $S$
for any $x \in U$.  Note that here $S$ is uniquely defined by $\rho$
up to conjugacy and $\dim(G \cdot x) = \dim(G) - \dim(S)$
for every $x \in U$.

\begin{lemma} \label{lem.prel0}
Let $G$ be a semisimple linear algebraic group and $G \to \GL(V)$ be a
finite-dimensional representation of $G$. Let $\CC [V]$ be the ring of
polynomial functions on $V$ and let $\CC (V)$ be the field of rational
functions. Then

\smallskip
(a) $\mathbb C(V)^G$ is the field of fractions of
$\mathbb C[V]^G$. In other words, every $G$-invariant rational function on
$V$ is a ratio of two $G$-invariant polynomials.

\smallskip
(b) The dimension of the GIT quotient $\mathbb P(V)\GITquot G$ equals
$\dim(V) - \dim(G) + \dim(S) - 1$.

\smallskip
(c) If $S$ is reductive, then $\mathbb P(V)\GITquot G \neq \emptyset$.
\end{lemma}

Note that $\dim(\mathbb P(V)\GITquot G) = -1$ if and only if
$\mathbb P(V)\GITquot G = \emptyset$.

\begin{proof} Part (a) is an easy consequence of the fact that
the polynomial ring $\mathbb C[V]$ is a unique factorization domain;
see, e.g.~\cite[Lemma 1]{popov} or \cite[Theorem 3.3]{popov-vinberg}.

\smallskip
(b) Recall that by definition, $\mathbb P(V)\GITquot G = \Proj \mathbb C[V]^G$.
Thus
\[ \dim \, P(V)\GITquot G = \trdeg_{\mathbb C} \mathbb C[V]^G - 1 =
\trdeg_{\mathbb C} \mathbb C(V)^G - 1 \, . \]
By Rosenlicht's theorem~\cite[Theorem 2]{rosenlicht}, the
transcendence degree of $\CC (V)^{G}$ is given by
\[ \trdeg_{\mathbb C} \mathbb C(V)^G = \dim(V) - \dim (G \cdot x) =
\dim(V) - (\dim(G) - \dim(G_x)), \]
where $x \in V$ is a point in general position; see also
\cite[Corollary, p.~156]{popov-vinberg}. Now recall that
$G_x$ is conjugate to $S$, so
$\dim(G_x) = \dim(S)$, and part (b) follows.

\smallskip
(c) Suppose $S$ is reductive. Then by a theorem of
Popov~\cite[Theorem 1]{popov}, there exists a $0 \neq x \in V$ such that
the orbit $G \cdot x$ is closed. By a theorem of Hilbert, $G$-invariant
polynomials separate closed orbits in $V$; (see,
e.g.,~\cite[Corollary 1.2]{mfk}).
In particular, there exists an $p \in \mathbb C[V]^G$ such that
$p(x) \neq p(0)$.  This shows that $\mathbb C[V]^G \neq \mathbb C$
and thus $\mathbb P(V)\GITquot G = \Proj \, \mathbb C[V]^G \neq \emptyset$.
\end{proof}

\begin{corollary} \label{cor1}
{\rm (}cf. \cite[\S 2, Propositions 2 and 3]{sato-kimura}{\rm )}
Let $G$ be a semisimple linear algebraic group and $G \to \GL(V)$
be a linear representation of $G$.  Then the following conditions
are equivalent.

\smallskip
(a) $G$ has a dense orbit in $V$,

\smallskip
(b) $\mathbb C(X)^G = \mathbb C$,

\smallskip
(c) $\mathbb C[X]^G = \mathbb C$,

\smallskip
(d) The GIT quotient $\mathbb P(V)\GITquot G$ empty.
\end{corollary}

\begin{proof} The implications
(a) $\Longrightarrow$ (b) $\Longrightarrow$ (c) $\Longrightarrow$ (d)
are obvious.

(d) $\Longrightarrow$ (a). Assume that the
GIT quotient $\mathbb P(V)\GITquot G$ empty, i.e., its dimension is $-1$.
Then by Lemma~\ref{lem.prel0}(b), $\dim(V) = \dim(G) - \dim(S)$.
On the other hand, for $x \in V$ in general position,
$\dim(G \cdot x) = \dim(V)$. Thus shows that
$G \cdot x$ is dense in $V$, as desired.
\end{proof}

\begin{corollary} \label{cor2}
{\rm (}\cite[\S 3, Proposition 1]{sato-kimura}{\rm )}
Let $G \to \GL(V)$ be a finite-dimensional representation
of a semisimple linear algebraic group $G$.  If $\dim(V) > \dim(G)$, then
$\mathbb P(V)\GITquot G \neq \emptyset$.
\end{corollary}

\begin{proof} If $\dim(V) > \dim(G)$, then by Lemma~\ref{lem.prel0}(b),
$\dim(\mathbb P(V)\GITquot G) \geqslant \dim(S) \geqslant 0$, i.e.,
$\mathbb P(V)\GITquot G \neq \emptyset$.
\end{proof}

\subsection*{The index of a representation}
Let $\rho \colon H \to \GL_d$ be a faithful finite-dimensional
representation of a simple complex linear algebraic
group $H$.  Let $\operatorname{Lie}(H)$
be the lie algebra of $H$ and
\[ \rho^* = d \rho_{| e} \colon \Lie(H) \to \Lie(GL_d) = \Mat_d \]
be the induced map of Lie algebras.
Following E. M.~Andreev, E.~B.~Vinberg and A.~G.~Elashvili~\cite{ave},
we define the index $l(\rho)$ by the formula
\[ l(\rho) = \frac{\Tr \! \big( \rho^*(h)^2 \big)}{\Tr \!
\big( \operatorname{ad}(h)^2 \big)} \, , \]
where $h \in \Lie(H)$. Note that the right hand side of the formula
is independent of the choice of $h$, as long as
$\Tr \! \big( \operatorname{ad}(h)^2 \big) \neq 0$.  It is clear from
this definition that the index is additive, i.e.,
\begin{equation} \label{e.additivity}
l(\rho_1 \oplus \rho_2) = l(\rho_1) + l(\rho_2) \, .
\end{equation}

\begin{example} \label{ex.index} Consider the natural
representation $\rho_{\rm nat}$ of $H = \SL_d$ on $V = \mathbb C^d$. Take
\[ \text{$h = \operatorname{diag}(\lambda_1, \dots, \lambda_d) \in
\mathfrak{sl}_d$, where $\lambda_1 + \dots + \lambda_d = 0$.} \]
Here $\mathfrak{sl}_{d}= \Lie(\SL_d)$.
Then $\rho_{\rm nat}^*(h) = h$, $\Tr(h^2) = \sum_{i = 1}^d \lambda_i^2$ and
\begin{gather}
\nonumber
\Tr \left(\operatorname{ad}(h)^2 \right) = \sum_{i \neq j} (\lambda_i - \lambda_j)^2 =
\sum_{i \neq j} (\lambda_i^2 + \lambda_j^2 - 2 \lambda_i \lambda_j) = \\
\nonumber
2(d-1) \sum_{i = 1}^d \lambda_i^2 - 2 \big( (\sum_{i = 1}^d \lambda_i)^2
 -  \sum_{i = 1}^d \lambda_i^2 \big) =
2d \sum_{i = 1}^d \lambda_i^2.
\end{gather}
We conclude that $l(\rho_{\rm nat}) = \dfrac{1}{2d}$.
\end{example}

\section{Proof of Theorem~\ref{thm.GIT}}\label{sec:proof}

\textbf{Proof of case (a):} We claim for every $v \in V$
the closure of the orbit
$\SL(V_n) \cdot v$ in $V$ contains $0$.
If we can prove this claim, then clearly
$\mathbb C[V]^G = \mathbb C$, and thus $\mathbb P(V)\GITquot G$ is empty; see
Corollary~\ref{cor1}.

To prove the claim,
let $U = V_1 \otimes \dots \otimes V_{n-1}$, let $W=V_{n}$, and
write $\sum_{i = 1}^m u_i \otimes w_i$, where $u_1, \dots, u_m$ form
a basis of $U$ and $w_{i}\in W$.
By our assumption, \[ m = d_1 \cdots d_{n-1} < d_n = \dim(V_n) \, . \]
Hence, we can choose a basis $f_1, \dots, f_{d_n}$ of $V_n$ such that
$w_1, \dots, w_m \in \Span(f_1, \dots, f_{d_n-1})$. Now define
the $1$-parameter subgroup
$\lambda \colon \mathbb G_m \to \SL(V_n)$ by $\lambda(t) f_j = t f_j$
for $j = 1, \dots, d_n - 1$ and $\lambda(t) f_n = t^{-d_n + 1} f_n$.
Then $\lambda(t) \cdot v = t v \to 0$, as $t \to 0$ and hence,
$0$ lies in the closure of $\SL(V_n) \cdot v$ in $V$, as claimed.

\smallskip
\textbf{Proof of case (b):} Let $U = V_1 \otimes \dots \otimes V_{n-1}$ and let $W=V_{n}$. By our assumption
$\dim(U) = \dim(W)$. Note that
\begin{equation} \label{e.main(2)}
\SL(W) \subset G \subset \SL(U) \times \SL(W).
\end{equation}
Identify $U$ with $W^*$ and thus $V = U \otimes W$ with the space
of $n \times n$-matrices $\Mat_n$, where $\SL(U) \simeq \SL_n$ acts
by multiplication on the left and $\SL(W) \simeq \SL_n$
acts by multiplication on the right. Let $f \colon U \otimes W =
\Mat_n \to \mathbb C$ be the determinant map. Then $f$ is invariant under
$\SL(U) \times \SL(W)$ and hence, under $G$; see~\eqref{e.main(2)}.
Thus shows that $\mathbb C[f] \subset \mathbb C[V]^G$. On the other hand,
\[ \mathbb C[V]^{G}  \subset   \mathbb C[V]^{\SL(W)} =
\mathbb C[\Mat_n]^{\SL_n} = \mathbb C[f] \, . \]
Thus $\mathbb C[V]^{G} = \mathbb C[f]$ is a polynomial ring in one variable.
Consequently, \[ \mathbb P(V)\GITquot G = \Proj \mathbb C[f] = \mathbb P^0 \]
is a single point, as claimed.

\smallskip

\textbf{Proof of case (c):} Let $W=V_{1}\otimes \dotsb \otimes
V_{d_{n-1}}$, let $H=SL_{d_{1}}\times \dotsb SL_{d_{n-1}}$, let $V' =
V_{d_1} \otimes \dots \otimes V_{d_{n-1}} \otimes V_{P - d_n}$, and
let $G' = \SL_{d_1} \times \dots \times \SL_{d_{n-1}} \times
\SL_{P-d_n}$, where $P = d_1 \dots d_{n-1}$.
By \cite[Lemma 2(a)]{littelmann}, $V \GITquot G$ is
isomorphic to $V' \GITquot G'$. Thus
\[ \dim(\PP(V) \GITquot G) = \dim(V \GITquot G) - 1
 = \dim(V' \GITquot G') - 1 = \dim(\PP(V') \GITquot G') \, , \]
as claimed.

\smallskip
\textbf{Proof of case (d):} Our argument will rely on the description of stabilizers
in general position in irreducible representations $\rho \colon G \to \GL(V)$
of a semisimple group $G$, satisfying the condition that the index
\begin{equation} \label{e.elashvili}
\text{$l(\rho_{| \, H}) \geqslant 1$ for every simple normal
subgroup $H$ of $G$,}
\end{equation}
due to Elashvili~\cite{elashvili}. In order to apply this description
to our representation
of $G = \SL(V_1) \times \dots \times \SL(V_n)$ of
$G$ on $V = V_1 \otimes \dots \otimes V_n$ (which we will denote by $\rho$),
we need to check that condition~\eqref{e.elashvili} is satisfied
for this representation.

The simple normal subgroups of $G$ are $H = \SL(V_i)$ for $i = 1, \dots, n$.
Clearly the restriction $\rho_{| \, \SL(V_1)}$ is isomorphic to the direct
sum of $\dim(V_2 \otimes \dots \otimes V_n) = d_2 \dots d_n$ copies
of the natural representation $\rho_{\rm nat}$ of $\SL(V_1)$.
As we saw in Example~\ref{ex.index}, $l(\rho_{\rm nat}) = \dfrac{1}{2d_1}$.
Hence, by~\eqref{e.additivity},
$l \big( \rho_{| \, \SL(V_1)} \big) = \dfrac{d_2 \dots d_n}{2d_1}$.
Similarly
$l \big( \rho_{| \, \SL(V_i)} \big) =
\dfrac{d_1 \cdots d_{i-1} d_{i+1} \cdots d_n}{2d_i}$ for
any $i = 1, 2, \dots, n$.
Since $d_1 \leqslant \cdots \leqslant d_n$, the smallest of these indices is
$l(\rho_n) = \dfrac{d_1 \dots d_{n-1}}{2d_n}$. The assumption of part (d),
that $d_n \leqslant \dfrac{1}{2} d_1 \cdots d_{n-1}$ is thus equivalent
to~\eqref{e.elashvili}.

Under this assumption~\cite[Theorem 9]{elashvili} asserts that
the connected component $S^0$ of the stabilizer $S$ in general position
for the action of $G$ on $V$ is as follows:
\[ S^0 \simeq \begin{cases}
\text{$(\mathbb G_m)^{d-1}$,
if $n = 3$, $(d_1, d_2, d_3) = (2, d, d)$ and $d \geqslant 3$,} \\
\text{$(\mathbb G_m)^{2}$,
if $n = 3$ and $(d_1, d_2, d_3) = (2, 2, 2)$,} \\
\text{$\{ 1 \}$,  otherwise.}
\end{cases} \]
In all cases, $S^0$ is reductive, and hence, so is $S$. By
Lemma~\ref{lem.prel0}(c), we conclude that $\mathbb P(V)\GITquot G$
is non-empty. Moreover, using the formula
$\dim(\mathbb P(V)\GITquot G) =  \dim(V) - \dim(G) + \dim(S) - 1$
of Lemma~\ref{lem.prel0}(c) and remembering that
$\dim(S) = \dim(S^0)$, we readily check that
\[ \dim(\mathbb P(V)\GITquot G) =
\begin{cases}
\text{$d - 3$, if $n = 3$, $(d_1, d_2, d_3) = (2, d, d)$, $d \geqslant 3$,} \\
\text{0, if $n = 3$ and $(d_1, d_2, d_3) = (2, 2, 2)$, and} \\
\text{$\dim(V) - \dim(G) - 1$ in all other cases.}
\end{cases} \]

\qed

\begin{remark} \label{rem.gkz}
Suppose $1 \leqslant d_1 \leqslant \dots \leqslant d_n$ and
\begin{equation} \label{e.c}
d_n \leqslant d_1 \cdots d_{n-1}/2.
\end{equation}
Theorem~\ref{thm.GIT}(c) tells us that then there exists
a non-constant $G=\SL(V_1) \times \dots \times \SL(V_n)$-invariant
polynomial on $V =V_1 \otimes \dots \otimes V_n$;
see Corollary~\ref{cor1}. Here, $\dim(V_i) = d_i$, as before.
Equivalently, there exists a homogeneous $G$-invariant polynomial of
degree $\geqslant 1$.
It is natural to try to exhibit such a polynomial explicitly.

It is well known that the hyperdeterminant
$h_{d_1, \dots, d_n} \colon V \to \mathbb C$
of format $d_1 \times \dots \times d_n$ is a homogeneous $G$-invariant
polynomial. However, $h_{d_1, \ldots, d_n} \neq 0$ if and only
if \[ d_n \leqslant d_1 + \dots + d_{n-1} - (n-2); \]
see~\cite[p. 446, Theorems 1.3 and Proposition 1.4]{gkz}. This inequality
is considerably weaker than~\eqref{e.c}. In other words, for many dimension
vectors $\dvec$ satisfying~\eqref{e.c}, the hyperdeterminant is identically
zero.

We also note that for any dimension vector $\dvec $ and an integer
$k \geqslant 1$, a procedure due to G.~Gour and N.~R.~Wallach~\cite{gw13}
produces
a basis for the vector space $(\CC[V]^G)_k$ of homogeneous $G$-invariant
polynomials of degree $k \geqslant 1$. However, when $\CC[V]^G \neq (0)$,
it is not a priori clear for which $k$, $(\CC[V]^G)_k \neq (0)$
\footnote{Gour and Wallach show that
$(\CC[V]^G)_k = (0)$ unless $k$ is divisible by the least
common multiple $l= \lcm(d_1, \dots, d_n)$. Thus we only need to consider
$k$ of the form $lq$, with $q = 1, 2, 3, \ldots$.}.
\end{remark}

\section{Proof of the first part of
Theorem~\ref{thm:main.theorem,numerical.formulation}}
\label{sec:analysis}

For each dimension vector $\dvec = (d_1, \dots, d_n)$,
with $d_1 \leqslant \dots \leqslant d_n$,
the recursive algorithm provided by
Theorem~\ref{thm.GIT} brings us to some terminal dimension vector
$\evec = (e_1, \dots, e_n)$, where
$e_1 \leqslant \ldots \leqslant e_n$.
The GIT quotient $\mathbb P(V)\GITquot G$ is empty
if the algorithm terminates on case (a), and non-empty if
the algorithm terminates on case (b) or (d).

\begin{remark} \label{rem.ones}
Recall from the Introduction (see footnote~\ref{footnote:
allow di=1 but need two non-trivial di}, pg.~\ref{page one}) that our standing assumption
is that $n \geqslant 2$ and $d_{n-1} \geqslant 2$. We now observe
that the terminal vector $\evec$ also satisfies these conditions.
Obviously, $n$ does not change; our claim is that
there are at most $n-2$ ones among the integers $e_1, \dots, e_n$.  Indeed,
in each recursion step, the number of 1s in the list of dimensions can
increase by one, but with $n-2$ 1s, we will either be in case (a) or
(b), and the recursion terminates.
\end{remark}

Let us now define a new
dimension vector $\avec = (a_1, \ldots, a_m)$ by removing all 1s from
$(e_1, \dots , e_n)$.

\begin{lemma}\label{lem.recursion}
Suppose that a terminal vector $\evec  = (e_1, \dots, e_n)$
satisfying the conditions for case (a), (b), or (d),
is obtained from $\dvec  = (d_1, \dots, d_n)$
by repeatedly performing the transformation
of Theorem~\ref{thm.GIT} (c).  Assume further that
$\avec = (a_1, \dots, a_m)$ is obtained from $\evec$ by removing all 1s,
as above.  Then

\smallskip
(a) $\Delta(\dvec ) = \Delta(\evec) = \Delta(\avec)$,

\smallskip
(b) $\gmax(\dvec ) = \gmax(\evec) = \gmax(\avec)$, and

\smallskip
(c) $R(\dvec ) = R(\evec) = R(\avec)$.
\end{lemma}

Note that $m \geqslant 2$ by
Remark~\ref{rem.ones}. In particular, $\gmax(\avec)$ is well defined.

\begin{proof}
Let $\dvec'  = (d_1', \dots, d_n')$ be the set of dimensions obtained from $\dvec$ by a single application of the castling transformation of
Theorem~\ref{thm.GIT}(c).
Since $\Delta(d_1, \dots, d_n)$, $R(d_1, \dots, d_n)$ and $\gmax(d_1, \dots, d_n)$ are all symmetric in $d_1, \dots, d_n$, we may reorder $d_1', \dots, d_n'$ and thus assume that $d_i' = d_i$
for $i \leqslant n -1$ and $d_n' = P - d_n$, where $P = d_1 \dots d_{n-1}$. The lemma follows by showing that each of the three quantities is invariant under $\dvec \to \dvec'$ and under eliminating a 1.

(a) An easy calculation shows that $\Delta(\dvec ') =
P(P - d_n) - (P - d_n)^2 - \sum_{i = 1}^{n-1} (d_i^2 - 1)$
is equal to
\[ \Delta(\dvec ) = Pd_n - d_n^2 - \sum_{i = 1}^{n-1} (d_i^2 - 1) .
\]
The identity $\Delta(1,d_2,\dots,d_n) = \Delta(d_2,\dots,d_n)$
is immediate from the definition.

(b) It is again obvious from the definition
that $\gmax(1,e_2,\dots,e_n) = \gmax(e_2,\dots,e_n)$.
Since $m \geqslant 2$, this implies that $\gmax(\dvec ) = \gmax(\evec )$.
To show invariance under the transformation in
Theorem~\ref{thm.GIT}(c), we will show more generally
that $\gcd(d_{i_1}', \dots , d_{i_k}') =
\gcd(d_{i_1}, \dots , d_{i_k})$ for any $2 \leqslant k \leqslant n$
and $1 \leqslant i_1 < \cdots < i_k \leqslant n$. If $i_n \leqslant n-1$,
this is obvious, since $d_{i_k}' = d_{i_k}$ for each $k$. If $i_k = n$, then
\[ \gcd(d_{i_1}', \dots , d_{i_{k-1}}', d_n') = \gcd(d_{i_1}, \dots , d_{i_{k-1}}, - d_n + P) = \gcd(d_{i_1}, \dots , d_{i_{k-1}}, d_n) \, , \] since $P$ is divisible by $d_{i_k}$ for $1 \leqslant k \leqslant n-1$.

(c) We have that $R(d_1,\dots, d_n) = \Delta(d_1, \dots d_n) - n + 1 + \sum_{k=2}^n (-1)^k G_k(d_1,\dots,d_n)$. Since $\Delta$ and the set of GCDs for all $k$-tuples with $k \geqslant 2$ are invariant under $\dvec \to \dvec'$, $R$ also invariant. If $d_1 = 1$, then all GCDs involving $d_1$ are equal to 1, so we have that $\sum_{k=1}^n (-1)^k G_k(d_1,\dots,d_n)$ is equal to $\sum_{k=1}^n (-1)^k G_k(d_2,\dots,d_n) + \sum_{k=1}^n (-1)^k \binom{n-1}{k-1} = \sum_{k=1}^n (-1)^k G_k(d_2,\dots,d_n)$. Since also $\prod_i d_i = \prod_{i \ne 1} d_i$, $R$ is unchanged if we remove dimensions equal to 1.
\end{proof}

We now show that the value of $R= R(\dvec )$ predicts on which case
the algorithm in Theorem~\ref{thm.GIT} will terminate, and thus
whether or not the quotient is empty.

\begin{proposition}\label{prop:Rcases}
For the dimension vector $\dvec = (d_1, \dots, d_n)$,
the algorithm in Theorem~\ref{thm.GIT}
terminates on case (a) if and only
if $R<0$, on case (b) if and only if $R=0$ and on case (d) if
and only if $R>0$.
\end{proposition}
\begin{proof}
Since the algorithm always terminates on case
(a), (b), or (d), we need only show that $R(\dvec)$ is respectively
negative, zero, and positive in these three cases.
Let $\evec = (e_1, \dots, e_n)$ be the terminal vector,
with $e_1 \leqslant \ldots \leqslant e_n$, and
$\avec = (a_1,\dots, a_{m-1},a_m)$ be
obtained from $\evec$ by removing all $1$s, as
above. By Lemma~\ref{lem.recursion},
it suffices to show that $R = R(\avec)$ is negative, positive, and zero
in cases (a), (b) and (d), i.e., if
$e_n > e_1 \cdots e_{n-1}$, $e_n = e_1 \cdots e_{n-1}$ and
$e_n \leqslant {1 \over 2} e_1 \cdots e_{n-1}$, or equivalently, if
$a_m > a_1 \cdots a_{m-1}$, $a_m = a_1 \cdots a_{m-1}$ and
$a_m \leqslant {1 \over 2} a_1 \cdots a_{m-1}$, respectively.
For the proof below, it will
be useful to define $B_k$ to be the sum of all the terms in $G_k$
with $k$-tuples involving $a_m$ and $A_k = G_k(a_1,\dots, a_{m-1})$ to
be the sum of the remaining terms. Then $G_k = A_k + B_k$, and
\be
\label{RABeq}
R = a_1 \cdots a_{m} - a_m^2 + \sum_{k=1}^{m-1} (-1)^k (A_k - B_{k+1}) \; .
\ee
We consider the three cases in turn.

\smallskip
\noindent
{\bf Case (b):} $a_m = a_1\dotsb  a_{m-1}$.

Here, each term in $B_{k+1}$ is equal to the corresponding term in
$A_k$ obtained by omitting $a_m$ from the GCD, so we have $B_{k+1} =
A_k$ for $k \geqslant 1$. Since $a_1 \cdots a_{m} - a_m^2 = 0$ in this
case, equation (\ref{RABeq}) gives $R(d_1,\dots, d_n) = 0$.

\smallskip
\noindent
{\bf Case (a):} $a_m > a_1\dotsb  a_{m-1}$.

In this case, we can write $a_m =a_1 \cdots a_{m-1} + \alpha$ for some $\alpha > 0$. From (\ref{RABeq}) we have
\be
\label{RAB}
R(a_1,\dots, a_{m-1},a_m) =-\alpha^2 -\alpha a_1 \cdots a_{m-1} - \sum_{k \geqslant 1 \; odd} (A_k - B_{k+1}) + \sum_{k \geqslant 2 \; even} (A_k - B_{k+1})
\ee
For each term in $A_k$, there is a corresponding term in $B_{k+1}$ obtained by including $a_m$ in the $gcd$. Since $gcd(a_{i_1}, \dots, a_{i_k},a_m) \leqslant gcd(a_{i_1}, \dots, a_{i_k})$, we have $A_k \geqslant B_{k+1}$.
Making use of this for odd $k$ and assuming for now that $\alpha \geqslant 2$, we obtain from (\ref{RAB})
\bea
R(a_1,\dots, a_{m-1},a_m) &<& -2 a_1 \cdots a_{m-1} + \sum_{k \geqslant 2\; even} A_k \cr
& \leqslant & -2 a_1 \cdots a_{m-1} + \sum_{k \geqslant 2\; even} \binom{m-1}{k} a_{m-2} a_{m-1} \label{ineq}\\
& = & -2 a_1 \cdots a_{m-1} + (2^{m-2}-1) a_{m-2} a_{m-1} \cr
& < & (2^{m-2} -1 - 2 a_1 \cdots a_{m-3} ) a_{m-2} a_{m-1} \nonumber
\eea
where to obtain the third line, we use that $A_k$ has $\binom{m-1}{k}$ terms and that $\gcd(a_{i_1} \cdots a_{i_k})^2 \leqslant a_{m-k}^2 \leqslant a_{m-2} a_{m-1}$ for $k \geqslant 2$ since the GCD of $k$ integers chosen from $(a_1,\dots, a_{m-1})$ cannot exceed $a_{m-k}$, the $k$th largest number in this set. Since $a_i \geqslant 2$, the term in brackets in the final expression is negative, and we conclude that $R<0$.

Next consider the case where $\alpha = 1$. Here, all GCDs involving $a_m$ are equal to 1, so we have $B_k = \binom{m-1}{k-1}$ and the terms in (\ref{RAB}) involving $B$ are
\be
\sum_{k \geqslant 1} (-1)^{k-1} B_{k+1} = \sum_{k \geqslant 1} (-1)^{k-1} \binom{m-1}{k} = 1
\ee
Then from (\ref{RAB}), we get
\be
\label{RAB2}
R(a_1,\dots, a_{m-1},a_m) =- a_1 \cdots a_{m-1} - \sum_{k \geqslant 1 \; odd} A_k  + \sum_{k \geqslant 2 \; even} A_k
\ee
As above, we can now decompose $A_k = C_k + D_k$, where $D_k$ represents all the terms in $A_k$ with $k$-tuples involving $a_{m-1}$ and $C_k$ are the remaining terms. The same argument as before shows that $C_k \geqslant D_{k+1}$.
Starting from (\ref{RAB2}) and eliminating negative terms $-D_{2l+1}$ and $-(C_{2l-1} - D_{2l})$ we then have
\bea
R(a_1,\dots, a_{m-1},a_m) &<& - a_1 \cdots a_{m-1} + \sum_{k \geqslant 2 \; even} C_k \cr
& \leqslant & (2^{m-3} -1 - a_1 \cdots a_{m-3} ) a_{m-2} a_{m-1} \nonumber
\eea
where the calculation is the same as in (\ref{ineq}) but we end up with $2^{m-3}$ instead of $2^{m-2}$ since $C_k$ involve only $m-2$ $a_k$s. Again, the term in brackets in the final expression is negative, and we conclude that $R<0$.

\smallskip
\noindent
{\bf Case (d):} $a_m \leqslant {1 \over 2} a_1,\dots, a_{m-1}$.

Note that this is only possible if $m \geqslant 3$.
Starting from equation (\ref{RABeq}), we have
\bea
\label{d_ineq}
R &=& {1 \over 4} a_1^2 \cdots a_{m-1}^2 - ({1 \over 2} a_1 \cdots a_{m-1} - a_m)^2 + \sum_{k \geqslant 1}(-1)^k (A_k - B_{k+1})
\cr
&\ge& a_{m-1}^2 (a_1 \cdots a_{m-2} - 1) + \sum_{k \geqslant 1}(-1)^k (A_k - B_{k+1})
\cr
&>& a_{m-1}^2 (a_1 \cdots a_{m-2} - 1) - \sum_{k \geqslant 1 \; odd} A_k
\eea
where in the second line, we have used that the maximum value of $({1 \over 2} a_1 \cdots a_{m-1} - a_m)^2$ will be for $a_m = a_{m-1}$, and in the third line, we have discarded non-negative terms $(A_k - B_{k+1})$ for $k$ even and positive terms $B_{k+1}$ for $k$ odd. Since each of the $\binom{m-1}{k}$ GCDs contributing to $A_k$ is less than or equal to $a_{k-1}$, we have that
\begin{eqnarray*}
R &>& a_{m-1}^2 (a_1 \cdots a_{m-2} - 1) - \sum_{k \geqslant 1 \; odd} a_{m-1}^2 \binom{m-1}{k} \cr
&=& a_{m-1}^2 (a_1 \cdots a_{m-2} - 1 - 2^{m-2})
\end{eqnarray*}
Since $a_1, \dots, a_{m-2} \geqslant 2$, we see that $R > 0$
unless $(a_1,\cdots, a_{m-2}) = (2, \dots, 2)$.

For this case, with $(a_1,\cdots, a_m) = (2,\dots,2,a_{m-1},a_m)$, we
can calculate the second line of (\ref{d_ineq}) directly. Consider two cases.

Case 1: $a_m$ is even. Here $A_1 = 4(m-2) + a_{m-1}^2$,
$B_2 = 4(m-2) + \gcd(a_{m-1}, a_m)^2$, and $A_k = B_{k+1}$
for any $k \geqslant 2$. Thus
\[ R \geqslant a_{m-1}^2 (2^{m-2} - 1) - (A_1 - B_2)
 = a_{m-1}^2 (2^{m-2} - 2) + \gcd(a_{m-1}, a_m)^2 > 0. \]

Case 2: $a_m$ is odd. Here
\[ B_2 = (m-2) + \gcd(a_{m-1}, a_m)^2 = \gcd(a_{m-1}, a_m)^2
- \binom{m-1}{0} + \binom{m-1}{1} \]
and $B_{k+1} = \begin{pmatrix} m-1 \\ k \end{pmatrix}$ for any $k \geqslant 2$. Thus
\[ \sum_{k \geqslant 1} (-1)^k B_{k+1} =
- \gcd(a_{m-1}, a_m)^2 + \sum_{i \geqslant 0} (-1)^i \binom{m-1}{i} =
- \gcd(a_{m-1}, a_m)^2  \, .  \]
Moreover, $A_1 = 4(m-2) + a_{m-1}^2 = a_{m-1}^2 - 4 \begin{pmatrix} m-1 \\ 0
\end{pmatrix} +
4 \begin{pmatrix} m-1 \\ 1 \end{pmatrix}$ and
\[ A_k = 4 \binom{m-2}{k} + \gcd(2, a_{m-1})^2 \binom{m-2}{k-1}
= 4 \binom{m-1}{k} + (\gcd(2, a_{m-1})^2 - 4) \binom{m-2}{k-1}  \]
for any $k \geqslant 2$.  Thus
\begin{eqnarray*}
R & \geqslant & a_{m-1}^2 (2^{m-2} - 1) -
\sum_{k \geqslant 1} (-1)^k B_{k + 1} - A_1
+ \sum_{k \geqslant 2} (-1)^k A_k \cr
 & = & a_{m-1}^2 (2^{m-2} - 1)  + \gcd(a_{m-1}, a_m)^2 - a_{m-1}^2 +
4 \binom{m-1}{0} - 4 \binom{m-1}{1} +  \sum_{k \geqslant 2} (-1)^k A_k \cr
 & = & a_{m-1}^2 (2^{m-2} - 2) + \gcd(a_{m-1}, a_m)^2 + 4 \sum_{k \geqslant 0} (-1)^k \binom{m-1}{k} \cr
 &  & + (\gcd(2, a_{m-1})^2 - 4) \sum_{k \geqslant 2} (-1)^k \binom{m-2}{k-1}
\cr
 & = & a_{m-1}^2 (2^{m-2} - 2)  + \gcd(a_{m-1}, a_m)^2 +
(\gcd(2, a_{m-1})^2 - 4) \, .
\end{eqnarray*}
If $m \geqslant 4$, then the first term is $\geqslant 8$ and thus
$R > 0$. As we mentioned above,
the inequality $a_m \leqslant a_1 \cdots a_{m-1}/2$ forces $m$
to be $\geqslant 3$.  Thus we may assume that $m = 3$. Since
$a_1 = 2$, we have $a_3 \leqslant 2 \cdot a_2/2 = a_2$ and thus
$a_2 = a_3$. In this case $\gcd(a_{m-1}, a_m) =
\gcd(a_2, a_3) = a_3$. Substituting $m = 3$ into the above inequality,
and remembering that $a_1 = 2$ and $a_2 = a_3 \geqslant 2$ is odd,
we obtain
\[ R \geqslant a_2^2 (2^{3-2} - 2) + a_2^2 +
(\gcd(2, a_2)^2 - 4)  \geqslant 0 + 9 + (1- 4) > 0 \, , \]
as desired.
\end{proof}

\begin{proof}[Proof of
Theorem~\ref{thm:main.theorem,numerical.formulation}, first part]
Suppose $R(\dvec ) < 0$. By Proposition~\ref{prop:Rcases}
the algorithm of Theorem~\ref{thm.GIT} terminates on case (a) and
the quotient $\PP(V) \GITquot G$ is empty.

On the other hand, $R(\dvec ) \geqslant 0$, Proposition~\ref{prop:Rcases}
shows that the algorithm of Theorem~\ref{thm.GIT} terminates on case (b) or (d).
Theorem~\ref{thm.GIT} now tells us that the quotient $\PP(V) \GITquot G$
is no-empty. This proves the first assertion of
Theorem~\ref{thm:main.theorem,numerical.formulation}.
\end{proof}

\section{Conclusion of the proof of
Theorem~\ref{thm:main.theorem,numerical.formulation}}
\label{sect.proof-main}

We will prove the remaining statements in
Theorem~\ref{thm:main.theorem,numerical.formulation} using the
following proposition.

\begin{proposition} \label{prop.terminal} Suppose that we perform
the recursive procedure of Theorem~\ref{thm.GIT}, starting with
the dimension vector $\dvec  = (d_1, \dots, d_n)$. Here, as always,
$d_1 \leqslant \cdots \leqslant d_n$, $n \geqslant 2$ and $d_{n-1} \geqslant 2$.
Denote the terminal dimension vector by $\evec  = (e_1, \dots, e_n)$.

\begin{enumerate}
\item
If $\Delta(\dvec ) < -5$, then $e_1 \dots e_{n-1} \leqslant e_n$.
That is, the recursion in Theorem~\ref{thm.GIT} terminates
on case (a) or (b).  The quotient $\mathbb{P}(V) \GITquot  G$
is empty or a single point.

\item
If $-5 \leqslant \Delta(\dvec ) < -2$, then
$e_1 \dots e_{n-1} = e_n$. That is, the recursion
in Theorem~\ref{thm.GIT} terminates on case (b).
The quotient $\mathbb{P}(V) \GITquot  G$ is a single point.

\item
If $\Delta(\dvec ) = -2$, then the recursion in
Theorem~\ref{thm.GIT} terminates on case (d) with
$\evec  = (1,\dots,1, 2, a, a)$ for some $a \geqslant 2$.
Here $a = \gmax( \dvec ) \geqslant 3$.
The quotient $\mathbb{P}(V) \GITquot  G$ is a point if $a = 2$ and has
dimension $a - 3$ if $a \geqslant 3$.

\item
If $\Delta(\dvec ) > -2$, then the recursion in
Theorem~\ref{thm.GIT} terminates on case (d) with
$\evec  \neq (1,\dots, 1, 2, a, a)$ for any
$a \geqslant 2$.  In this case $\Delta(\dvec ) \geqslant 2$, and
the quotient $\mathbb{P}(V) \GITquot  G$ has dimension $\Delta(\dvec )$.
\end{enumerate}
\end{proposition}

\begin{proof}
Let $\avec  = (a_1, \ldots, a_m)$ be obtained
by removing all 1s from $\evec = (e_1, \dots , e_n)$,
as in the previous section.
By Remark~\ref{rem.ones}, $m \geqslant 2$.
By Lemma~\ref{lem.recursion}, $\Delta(\dvec ) =
\Delta(\evec ) = \Delta(\avec )$
and $\gmax( \dvec ) = \gmax( \evec ) = \gmax( \avec )$.
Moreover, one readily sees that $\avec $ is also a terminal vector
for the recursive procedure of Theorem~\ref{thm.GIT}, and that
$\evec $ and $\avec $ correspond to the same
terminal case in Theorem~\ref{thm.GIT}, i.e., case (a), (b) or (d).
Thus, for the purpose of proving Proposition~\ref{prop.terminal}, we may
replace $\dvec $ by $\avec $. That is, we may assume that
$\dvec  = (d_1, \dots, d_n)$ is terminal, $2 \leqslant d_1 \leqslant
\cdots \leqslant d_n$, and $n \geqslant 2$.

Set $P = d_1 \cdots d_{n-1}$.  To prove the proposition, we will show that

\smallskip
(i) if $\dvec $ corresponds to case (a), i.e. $d_n > P$,
then $\Delta(\dvec ) < - 5$,

\smallskip
(ii) if $\dvec $ corresponds to case (b), i.e. $d_n = P$,
then $\Delta(d_1, \dots, d_n) < - 2$,

\smallskip
(iii) if $\dvec  = (2, a, a)$, then $\Delta( \dvec ) = -2$
and $\gmax( \dvec ) = a$.

\smallskip
(iv) if $\dvec $ corresponds to case (d), i.e. $d_n \leqslant P/2$,
and moreover $\dvec  \ne (2, a, a)$ for any integer $a \geqslant 2$,
then $\Delta(\dvec ) \geqslant 2$.

\smallskip
If we can establish (i) - (iv), Proposition~\ref{prop.terminal} will follow
directly from Theorem~\ref{thm.GIT}.

To prove (i), let us fix ${d}_1, \dots, {d}_{n-1}$ and
view
\[ \phi(x) = \Delta({d}_1, \dots, {d}_{n-1}, x)
= x P - d_1^2 - \ldots - d_{n-1}^2  -
x^2 + n -1 \]
as a quadratic polynomial in $x$. Note that $f'(x) = P - 2x$, so
$\phi(x)$ is increasing for $x \leqslant P/2$ and decreasing
for $x \geqslant P/2$.  In particular, if ${d}_n \geqslant P + 1$, then
\[
\Delta(\dvec ) = \phi({d}_n) \leqslant \phi(P + 1) =
-P - 1 -\sum_{i=1}^{n-1} (d_i^2 - 1) \, .
\]
Since $n \geqslant 2$, each $d_i \geqslant 2$ and in particular,
$P \geqslant {d}_1 \geqslant 2$. This yields
\[ \Delta(\dvec ) \leqslant - 2 - 1  - (2^2 - 1) = -6, \]
as claimed.

To prove (ii), assume $d_n = P$. Then
$\Delta(\dvec ) = \phi(P) = - \displaystyle{\sum_{i=1}^{n-1}}
(d_i^2 - 1) < - 2$,
since $n \geqslant 2$ and each $d_i \geqslant 2$.

(iii) is easy: $\Delta(2, a, a) = 2 a^2 - 4 - a^2 -
a^2 + 3 - 1= -2$ by the definition of $\Delta$,
and $\gmax( 2, a, a) = a$ by the definition of $\gmax$.

To prove (iv), assume that
$2 \leqslant d_1 \leqslant \dots \leqslant d_{n-1} \leqslant
{d}_{n} \leqslant P/2$.
We are interested in the value of $\Delta(\dvec ) = \phi({d}_n)$.
Since $f'(x) \leqslant 0$ for any $x$ in the interval
$[{d}_{n-1}, P/2]$, we have
\[ \Delta(\dvec ) = \phi(d_n) \geqslant \phi({d}_{n-1}) =
{d}_{n-1}P - {d}_1^2 - ... - {d}_{n-2}^2 - 2 d_{n-1}^2 + n -1 \, .  \]
Remembering that $P = d_1 \cdots d_{n-1}$ and
${d}_i \leqslant {d}_{n-1}$ for $i = 1, \ldots, n-1$,
we conclude that
\begin{equation} \label{e.delta2}
\Delta(\dvec ) \geqslant {d}_{n-1}P - n ({d}_{n-1})^2 + n-1 =
{d}_{n-1}^2 ({d}_{1} \dots {d}_{n-2} - n) + n-1 \, .
\end{equation}
By our assumption, $n \geqslant 3$ and $(d_1, \dots, d_n) \neq (2, d, d)$
for any $d \geqslant 2$. Let us now consider two cases.

\smallskip
Case 1: $n \geqslant 4$. In this case
$d_{1} \dots d_{n-2} - n \geqslant 2^{n - 2} - n \geqslant 0$
and~\eqref{e.delta2} tells us that
$\Delta(\dvec ) \geqslant n - 1 \geqslant 3$.

\smallskip
Case 2: $n = 3$ but $(d_1, d_2, d_3) \neq (2, d, d)$ for any $d \geqslant 2$.
Here our assumption that $d_3 \leqslant P/2 = d_1 d_2/2$ implies
$d_1 \geqslant 3$.  In this case~\eqref{e.delta2} yields
$\Delta(\dvec ) \geqslant d_2^2 (d_1 - 3) + 2$,
and $d_1 \geqslant 3$ implies $\Delta(\dvec  ) \geqslant 2$.
This completes the proof of (iv) and thus of Proposition~\ref{prop.terminal}.
\end{proof}

\begin{proof}[Proof of
Theorem~\ref{thm:main.theorem,numerical.formulation}, second part]
To prove the second assertion of Theorem~\ref{thm:main.theorem,numerical.formulation}, let us consider
three cases, where $\Delta(\dvec ) < -2$, $\Delta (\dvec ) = -2$, and
$\Delta (\dvec ) > -2$, respectively.

When $\Delta (\dvec ) > -2$, Proposition~\ref{prop.terminal}(4) tells us
that the recursion terminates on case (d) and
the dimension of $\PP(V) \GITquot G$ is $\Delta(\dvec ) \geqslant 2$.
In this case $R > 0$ By Proposition~\ref{prop:Rcases}.

When $\Delta (\dvec ) = -2$, we are in case (3)
of Proposition~\ref{prop.terminal}.
Here the recursion terminates on case (d) and
the dimension of $\PP(V) \GITquot G$ is
$(\gmax(\dvec) - 3)$ for $\gmax(\dvec) \geqslant 3$ and
$0$ otherwise.

Finally, when $\Delta(\dvec ) < -2$ we are in case (1) or (2) of
Proposition~\ref{prop.terminal}.  The proposition tells us that that
the recursion terminates on case (a) or (b) of Theorem~\ref{thm.GIT}.
By Proposition~\ref{prop:Rcases},
the recursion terminates in case (a) if $R < 0$ and in case (b) if $R = 0$.
Combining this with Theorem~\ref{thm.GIT}, we see that
$\PP(V) \GITquot G$ is a single point if $\Delta(\dvec ) < -2$ and $R = 0$
and is empty if $\Delta (\dvec ) < -2$ and $R = 0$, as desired.
\end{proof}

\section{Examples}

To conclude, we describe a few explicit results implied by Theorems \ref{thm:main.theorem,numerical.formulation} and \ref{thm.GIT}.

\begin{corollary} \label{cor.d=2}
For dimension vectors $(2,d_2,d_3)$, with $2 \leqslant d_2 \leqslant d_3$,
the quotient $\mathbb P(V) \GITquot  G$ is non-empty if and only if

\smallskip
(i) $(d_1,d_2,d_3) = (2,b,b)$ for $b \geqslant 2$
or

\smallskip
(ii) $(d_1, d_2, d_3) = (2,kb,(k+1)b)$ for positive
integers $k,b$ with $kb > 1$.

\smallskip
In case (i), the quotient $\mathbb P(V)\GITquot G$ is
of dimension $\max(b - 3, 0)$.
In case (ii), $\mathbb P(V)\GITquot G$ is a single point.
\end{corollary}

\begin{proof} First assume that $(d_1, d_2, d_3)$ is as in (i) and (ii).
Note that $\Delta(2, d_2, d_3) = -(d_3 - d_2)^2 - 2$. In particular, in
case (i), $\Delta(d_1, d_2, d_3) = -2$, and the desired conclusion follows from
Proposition~\ref{prop.terminal}(3). In case (ii),
the recursive procedure of Theorem~\ref{thm.GIT} yields
\begin{equation} \label{e.terminal}
(2, kb, (k+1)b) \mapsto (2, (k-1)b, kb) \mapsto \dots
\mapsto (2, 2b, 3b) \mapsto (2, b, 2b) \, .
\end{equation}
The terminal triple $(2, b, 2b)$ is covered by Theorem~\ref{thm.GIT}(b),
for any $b \geqslant 1$, except that for $b= 1$, we should write it as
$(1, 2, 2)$, rather than $(2, 1, 2)$. (Note that we can also
check directly that $R(\dvec) = b^2 - 4 + \gcd(2, b)^2 > 0$ in case (i)
and $R(\dvec) = 0$ in case (ii).)

Conversely, suppose $\mathbb P(V)\GITquot G$ is non-empty
for some dimension vector
$(2, d_2, d_3)$. Denote the terminal triple by $\evec   = (e_1, e_2, e_3)$.
Then either $e_1 = 1$ and $e_2 = 2$, or $e_1 = 2$. Moreover, either
$e_1 e_2 = e_3$, as in Theorem~\ref{thm.GIT}(b) or $e_3 \leqslant e_1 e_2 /2$,
as in Theorem~\ref{thm.GIT}(d). This leaves us with

\smallskip
(1) $(e_1, e_2, e_3) = (2, b, b)$ or

\smallskip
(2) $(e_1, e_2, e_3) = (2, b, 2b)$, where $b \geqslant 2$ or

\smallskip
(3) $(e_1, e_2, e_3) = (1, 2, 2)$.

In cases (2) and (3), we recover
$(d_1, d_2, d_3) = (2, kb, (k+1) b)$, for $b \geqslant 2$ and $b = 1$,
respectively, by reversing the recursive procedure~\eqref{e.terminal}.
\end{proof}

\begin{remark}
For $\{n=3,d_1>2\}$ or $n \ge 3$, the naive dimension $\Delta(d_1, \dots d_n)$ considered as a function of $d_n$ is a downwards parabola that is positive
for $d_n = d_{n-1}$, increases to a maximum
at $d_n = {1 \over 2} d_1 \cdots d_{n-1}$, and then
decreases to $-2$ at some $d_* \in (P/2,P)$ where $P = d_1 \cdots d_{n-1}$.
Thus, by Proposition~\ref{prop.terminal}, the quotient is nonempty and has dimension $\Delta(d_1, \dots d_n)$
for all $d_n$ in the range $[d_{n-1},d_*)$. If $d_*$ is an integer,
the quotient is non-empty for $d_n = d_*$ and has dimension governed
by case (3) of Proposition~\ref{prop.terminal}. The remaining values
of $d_n$ for which the quotient is nonempty are a set of sporadic cases
with $d_* < d_n \leqslant P$ satisfying $R(\dvec) = 0$ for which the quotient is a point. We provide
a more detailed analysis of these sporadic cases for $n=3$ in~\cite{sam}.
\end{remark}

\section*{Acknowledgments} We are grateful to Omer Angel, Jason Bell,
Michel Brion, Shrawan Kumar, Samuel Leutheusser, Tomasz
Maci\c{a}\.{z}ek, Greg Martin, Vladimir L. Popov,
and Michael Walter for stimulating discussions.

\end{document}